\documentclass[a4paper]{article}

\usepackage{amssymb,amsmath, amsfonts}
\usepackage{authblk}
\usepackage{pst-grad} 

\usepackage{tikz-cd}
\usetikzlibrary{matrix,arrows,decorations.pathmorphing}

\usepackage{cmdtrack}
\usepackage{amsthm}
\usepackage{url}

\newcommand\id{\mbox{\bf 1}}
 \newtheorem{thm}{Theorem}[section]
 \newtheorem{prop}[thm]{Proposition}
 \newtheorem{lem}[thm]{Lemma}
 \newtheorem{cor}[thm]{Corollary}

\numberwithin{equation}{section}

\newtheorem{defn}{{Definition}}[section]
\usepackage{kbordermatrix}
\newtheorem{ex}{{Example}}[section]
\newcommand{\beaz}{\begin{eqnarray*}}
\newcommand{\eeaz}{\end{eqnarray*}}
\newcommand{\bea}{\begin{eqnarray}}
\newcommand{\eea}{\end{eqnarray}}

\begin{document}

\title{On the Faithfulness of $1$-dimensional Topological Quantum Field Theories}
\author{Sonja Telebakovi\' c  Oni\' c}
\affil{Faculty of Mathematics,
\\Studentski trg 16,\\ 11000 Belgrade, Serbia

\vspace{1ex}

\texttt{sonjat@matf.bg.ac.rs }}

\date{}
\maketitle

\vspace{-3ex}

\begin{abstract}

This paper explores $1$-dimensional topological quantum field theories. We separately deal with strict and strong $1$-dimensional topological quantum field theories. The strict one is regarded as a symmetric monoidal functor between the category of $1$-cobordisms and the category of matrices, and the strong one is a symmetric monoidal functor between the category of $1$-cobordisms and the category of finite dimensional vector spaces.
It has been proved that both strict and strong $1$-dimensional topological quantum field theories are faithful.

\end{abstract}

\vspace{.3cm}


\vspace{.5ex}

\noindent {\small {\it Keywords$\,$}: symmetric monoidal category, oriented manifold, cobordism, topological quantum field theory, Brauerian
representation, Kronecker product, commutation matrix
}

\section{Introduction}

\indent \indent The concept of topological quantum field theory goes back to the work of Witten (\cite{EW}). Mathematical axioms for topological quantum field theories are given by Atiyah (\cite{AM}). The categorical viewpoint is developed in Quinn's lectures (\cite{FK}). According to Atiyah and Quinn, an $n$-dimensional topological quantum field theory ($n$-TQFT) is a symmetric monoidal functor from the category of oriented $n$-cobordisms to the category of finite dimensional vector spaces over some field. Such functor associates a vector space with each closed oriented $(n-1)$-dimensional manifold and associates a linear map with each oriented $n$-cobordism. It is well known that there is a one-to-one correspodence between $2$-TQFTs and commutative Frobenius algebras (the best reference here is \cite[Theorem~3.3.2]{K1}). For another proof of this classical correspondence see \cite[Theorem~4.1]{J14}.

The faithfulness property of TQFTs is very important, since it provides a complete set of invariants for the classification of cobordisms. 
However, the faithfulness problem of TQFTs is seldom investigated in relevant literature. Inspired by the representation of a class of diagrammatic algebras given by Brauer (\cite{BR1}), Do\v sen and Petri\' c have obtained a result that demonstrates the faithfulness of a $1$-TQFT (see \cite[Section 14]{DP4}). This example of a faithful $1$-TQFT maps the $0$-dimensional sphere $S^{0}$ to a matrix Frobenius algebra, as explained in \cite[Section 5]{BPT}. It is shown by Petri\' c and the author (\cite{PT17}) that there is a faithful $2$-TQFT which corresponds to the commutative Frobenius algebra $\mathbb{QZ}_5\otimes Z(\mathbb{QS}_3)$, the tensor product of the group algebra and the center of the group algebra. One question remains unanswered: is there a faithful $n$-TQFT, for $n\geq3$?

The aim of this paper is to prove the faithfulness of all $1$-TQFTs. In Section 2, we summarize the relevant material on the category of $1$-cobordisms, and associate a $0-1$ matrix with each $1$-cobordism. This association is motivated by Brauer's matrix representation of the class of diagrammatic algebras (\cite{BR1}). A generalization of Brauer's representation given by Do\v sen and Petri\' c (\cite{DP3,DP4}) leads to a symmetric monoidal functor between the category of $1$-cobordisms and the category of matrices. In this way we obtain what will be referred to as the Brauerian functor. Our results on the faithfulness of this functor are presented in Section 3. In Section 4, we introduce the notion of a strict $1$-TQFT in order to use matrix techniques.
The main result of this section is that every strict $1$-TQFT, $F:1Cob \rightarrow Mat_{\mathbb{K}}$, mapping the null-dimensional manifold consisting of one point to a natural number $p\geq 2$, is faithful, since it coincides with the Brauerian functor up to multiplication by invertible matrices. Roughly speaking, the faithfulness of $F$ means that $F(K)$ is a complete set of algebraic invariants for $1$-cobordisms.
The last section is devoted to the study of strong $1$-TQFTs (symmetric strong monoidal functors between the category of $1$-cobordisms and the category of finite dimensional vector spaces over a chosen field). We extend our faithfulness result to the case of strong $1$-TQFTs.

\section{The Category $1Cob$}

\quad\,\,Objects of the category $1Cob$ are closed oriented null-dimensional manifolds, consisting of a finite number of points. Every particular point is associated with a sign, that represents its orientation. From now on, we will consider the objects of $1Cob$ as ordered pairs $(n,\varepsilon)$,
where $n=\{0,\ldots,n-1\}$, and $\varepsilon: n\rightarrow\{-1,1\}$.

The morphisms of $1Cob$ are the equivalence classes of $1$-cobordisms $\hfil\hfil\hfil$ $(M,f_{0}:
(n,\varepsilon_{0})\rightarrow M, f_{1}:
(m,\varepsilon_{1})\rightarrow M)$, where $M$ is a compact oriented ${1}$-dimensional manifold, such that its boundary $\partial M$ is a disjoint union of $\Sigma_0$, and $\Sigma_1$; $f_{0}$ is an orientation preserving embedding which image is $\Sigma_{0}$, while $f_{1}$ is an orientation reversing
embedding which image is $\Sigma_{1}$.
The manifolds $\Sigma_0$ and $\Sigma_1$ are called the ingoing and outgoing boundary of $M$, respectively. Two $1$-cobordisms $K=(M,f_0,f_1)$ and
$K'=(M',f'_0,f'_1)$ are equivalent, denoted by $K\sim K'$, if there is an orientation preserving homeomorphism $F:M\rightarrow M'$ such that the following diagram commutes.

\begin{center}
\begin{picture}(200,80)

\put(10,40){\makebox(0,0){$(n, \varepsilon_{0})$}}
\put(90,75){\makebox(0,0){$M$}}
\put(90,0){\makebox(0,0){$M'$}}
\put(170,40){\makebox(0,0){$(m, \varepsilon_{1})$}}
\put(38,63){\makebox(0,0){$f_0$}}
\put(142,63){\makebox(0,0){$f_1$}}
\put(38,17){\makebox(0,0){$f'_0$}}
\put(142,17){\makebox(0,0){$f'_1$}}
\put(100,40){\makebox(0,0){$F$}}

\put(30,50){\vector(2,1){47}} \put(30,30){\vector(2,-1){47}}
\put(150,50){\vector(-2,1){47}}
\put(150,30){\vector(-2,-1){47}} \put(90,60){\vector(0,-1){40}}
\end{picture}
\end{center}

The category $1Cob$ is a strict monoidal with respect to the sum on objects
$$(n,\varepsilon_{0})+(n',\varepsilon_{1})=(n+n',\varepsilon_{0}+\varepsilon_{1}),$$
where $\varepsilon_{0}+\varepsilon_{1}:n+n'\rightarrow
\{-1,1\}$ is defined by \begin{equation*}
 (\varepsilon_{0}+\varepsilon_{1})(x)=\begin{cases}
    \varepsilon_{0}(x), & \text{{if} $x\in n$}\\
    \varepsilon_{1}(x-n), & \text{{if} $x\not\in n,$}
  \end{cases}
\end{equation*}
and the operation of ``putting side by side'' on morphisms, denoted by $\otimes$.

The category $1Cob$ is also a symmetric monoidal with respect to the family of cobordisms $\tau_{n,m}: (n+m, \varepsilon_{0}+\varepsilon_{1})
\rightarrow (m+n, \varepsilon_{1}+\varepsilon_{0})$ corresponding to permutations on $n+m$. For example, $\tau_{3,2}$ is illustrated by the following picture

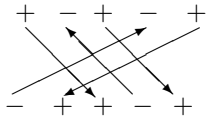
\begin{center}
\begin{picture}(75,40)

\put(5,32){\vector(1,-1){26}}
\put(35,32){\vector(1,-1){25}}
\put(45,7){\vector(-1,1){25}}
\put(0,7){\vector(2,1){50}}
\put(20,7){\line(2,1){50}}
\put(1,35){$+$}
\put(17,35){$-$}
\put(30,35){$+$}
\put(47,35){$-$}
\put(65,35){$+$}

\put(-3,0){$-$}
\put(15,0){$+$}
\put(30,0){$+$}
\put(45,0){$-$}
\put(60,0){$+$}

\put(22,8){\vector(-2,-1){3}}
\end{picture}
\end{center}

Every morphism of $1Cob$ is completely determined by a $1$-dimensional manifold $M$ and two sequences of its boundary points, denoted by $$(0,0),
(1,0),\ldots,(n-1,0),$$
$$(0,1), (1,1),\ldots,(m-1, 1),$$ where the first one corresponds to the ingoing boundary $\Sigma_{0}$, and the second one to the outgoing boundary $\Sigma_{1}$.

Every $1$-cobordism $K=(M, \Sigma_{0}, \Sigma_{1}):
(n,\varepsilon_{0})\rightarrow(m,\varepsilon_{1})$ induces the following equivalence relation $R_{K}$ on the set $(n\times\{0\})\cup(m\times\{1\})$. For
$(i,k), (j,l)$ elements of
$(n\times\{0\})\cup(m\times\{1\})$, we have that $$((i,k), (j,l))\in
R_{K}\ \mbox{iff}$$  $$\mbox{points}\ (i,k)\
\mbox{and}\ (j,l)\ \mbox{belong to the same connected component of}\ M.$$

Let $c_{K}$ denote the number of connected components of $K$, which are homeomorphic to the circle $S^{1}$. It is clear that for cobordisms
$K,L:(n,\varepsilon_{0})\rightarrow(m,\varepsilon_{1})$ the following proposition holds.

\begin{prop}\label{ekvkobordizmi}{The $1$-cobordisms $K$ and $L$ are equivalent iff $R_{K}=R_{L}$ and
$c_{K}=c_{L}$.}\end{prop}

Let $X$ be an arbitrary set, let $R\subseteq X^{2}$ be an equivalence relation on $X$ and let $p$ be the set
$\{0,1,\ldots,p-1\}$, $p\geq 2$.
Consider the following set of functions:
$$\mathcal{F}^{\ \!=}(R)=\{f: X\rightarrow p\ |\
(\forall x,y\in X)\ \bigl{(}(x,y)\in R\Rightarrow f(x)=f(y)\bigr{)}\}.$$
\begin{prop}\label{KONDOM}{{(}Do\v sen and Petri\' c  \cite{DP1}, \cite{DP2}}{)}\,
{If $R_{1},R_{2}\subseteq X^{2}$ are equivalence relations, then $R_{1}=R_{2}$ iff $\mathcal{F}^{\
\!=}(R_{1})=\mathcal{F}^{\ \!=}(R_{2})$}.
\end{prop}

We associate a non-zero matrix $A(K)$ of order $p^{m}\times p^{n}$
with each cobordism $K:(n,\varepsilon_{0})\rightarrow(m,\varepsilon_{1})$ in the following way.
Let $R_{K}\subseteq ((n\times\{0\})\cup(m\times\{1\}))^{2}$ be equivalence relation corresponding to the cobordism $K$.
The number of rows of $A(K)$ is equal to the number of functions $m\rightarrow p$.
Each of these functions can be envisaged as a sequence of length $m$ of elements of $\{0,1,\ldots,p-1\}$.
The set of these sequences may be ordered lexicographically so that $00\ldots 0$ is the first, and
$(p-1)(p-1)\ldots (p-1)$ is the last in this ordering. Since $(p^{m},\leq)$ is isomorphic to the set of these
sequences, the rows of $A(K)$ can be identified by functions from $m$ to $p$. Let $f_{i}:m\rightarrow p$
denote the function corresponding to the $i$-th row. Columns of $A(K)$ can be identified by functions from $n$ to $p$.
Let $g_{j}:n\rightarrow p$ denote the function corresponding to the $j$-th column.
Let $[g_{j},f_{i}]:(n\times\{0\})\cup(m\times\{1\})\rightarrow p$
be the function defined by $$[g_{j},f_{i}](x) = \left\{
	\begin{array}{ll}
		g_{j}(\pi_{0}(x)),  & \mbox{if}\ x \in n\times\{0\} \\
		f_{i}(\pi_{1}(x)), &  \mbox{if}\ x \in m\times\{1\}
	\end{array} \right.,$$ where $\pi_{0}:
n\times\{0\}\rightarrow n$ and $\pi_{1}: m\times\{1\}\rightarrow
m$ are bijections given by $\pi_{0}(u,0)=u$ and $\pi_{1}(v,1)=v$, respectively.
Element $A(K)[i,j]$ in the $i$-th row and $j$-th column of $A(K)$ is equal to $1$ iff $[g_{j},f_{i}] \in \mathcal{F}^{\
\!=}(R_{K})$, otherwise it is $0$.

\begin{ex}\label{prim2}{If $K:(4,\varepsilon_{0})\rightarrow(2,\varepsilon_{1})$
is $1$-cobordism illustrated by the following picture
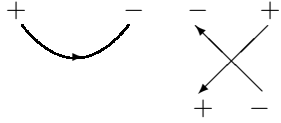
\begin{center}
\begin{picture}(100,40)

\qbezier(10,32)(30,7)(50,32)
\put(100,7){\vector(-1,1){25}}
\put(77,7){\line(1,1){25}}
\put(72,35){$-$}
\put(4,35){$+$}
\put(99,35){$+$}
\put(48,35){$-$}

\put(74,-2){$+$}
\put(95,-2){$-$}

\put(30,20){\vector(1,0){3}}
\put(79,9){\vector(-1,-1){3}}
\end{picture}
\end{center}

and $p=2$, the corresponding matrix $A(K)$ is equal to

\begin{changemargin}{-0.9cm}{0cm}

{\scriptsize
\renewcommand{\kbldelim}{(}
\renewcommand{\kbrdelim}{)}
\[
 \kbordermatrix{
    & 0000 & 0001 & 0010 & 0011 & 0100 &0101&0110&0111&1000&1001&1010&1011&1100&1101&1110& 1111 \\
    00 & 1 & 0 & 0 & 0 & 0 & 0 & 0 & 0 & 0 & 0 & 0 & 0 & 1 & 0 & 0 & 0\\
    01 & 0 & 0 & 1 & 0 & 0 & 0 & 0 & 0 & 0 & 0 & 0 & 0 & 0 & 0 & 1 & 0\\
    10 & 0 & 1 & 0 & 0 & 0 & 0 & 0 & 0 & 0 & 0 & 0 & 0 & 0 & 1 & 0 & 0\\
    11 & 0 & 0 & 0 & 1 & 0 & 0 & 0 & 0 & 0 & 0 & 0 & 0 & 0 & 0 & 0 & 1 \\}.
\]
}

\end{changemargin}

For example, element $A(K)[2,3]$ is equal to $1$, since the sequences $01$ and $0010$ corresponding to the second row and the third column ``match'' into the picture.

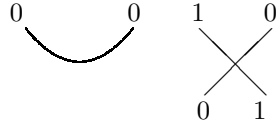
\begin{center}
\begin{picture}(100,40)

\qbezier(10,32)(30,7)(50,32)
\put(100,7){\line(-1,1){25}}
\put(77,7){\line(1,1){25}}
\put(72,35){$1$}
\put(4,35){$0$}
\put(99,35){$0$}
\put(48,35){$0$}

\put(74,-2){$0$}
\put(95,-2){$1$}

\end{picture}
\end{center}}
\end{ex}
\vspace{-5mm}

\section{Brauerian representation}

Let $Mat_{\mathbb{K}}$ be a category whose objects are natural numbers, and whose morphisms from $n$ to $m$ are all $m \times
n$ matrices over the field $\mathbb{K}$ of characteristic zero.
The composition of morphisms is matrix multiplication, and the identity morphism $n\rightarrow n$ is the identity matrix of order $n$. The category $(Mat_{\mathbb{K}}, \otimes, 1,
S_{n,m})$ is a symmetric strict monoidal with respect to the multiplication on objects, and the Kronecker product on morphisms, and the family of $mn
\times nm$ matrices $S_{n,m}$. The matrix $S_{n,m}$ is the matrix representation of the linear map $\sigma: \mathbb{K}^{n}\otimes
\mathbb{K}^{m}\rightarrow \mathbb{K}^{m}\otimes \mathbb{K}^{n}$
with respect to the standard ordered bases, defined on the basis vectors by $\sigma(e_{i}\otimes f_{j})=f_{j}\otimes e_{i}$.
We call $S_{n,m}$ the commutation matrix.

Let us now consider the following functor $B$ from $1Cob$ to $Mat_{\mathbb{K}}$. It is defined by $B(n, \varepsilon)=p^{n}$ on objects, where $p\geq 2$. For morphism $K:
(n,\varepsilon_{0})\rightarrow (m,\varepsilon_{1})$ of $1Cob$, it is given by $B(K)=p^{a}\cdot A(K)$, where $a$
is the number of circular components of $K$, and $A(K)$ is $(0,1)$-matrix of order $p^{m} \times p^{n}$ associated with $K$.

It can be proved that the following proposition holds.
\begin{prop}\label{funktorijalnost}{$B$ is a strict symmetric monoidal functor from $\hfill\hfill\hfill\hfill\hfill\hfill\hfill\hfill\hfill\hfill$ $(1Cob, \otimes, (0,\varepsilon),
\tau_{n,m})$ to $(Mat_{\mathbb{K}}, \otimes, 1,
S_{n,m})$.}\end{prop}
The proof of this proposition can be adapted from \cite[Section~5, Proposition~4]{DP2}.

The functor $B$ is related to the matrix representation of a class of diagrammatic algebras given by Brauer (see \cite{BR1},\cite{DP3} for more details). Therefore, we call $B$ Brauerian functor.

\begin{prop}\label{vernost}{$B$ is faithful.}\end{prop}

\begin{proof} Suppose that $K, L:
(n,\varepsilon_{0})\rightarrow(m,\varepsilon_{1})$ are two morphisms of $1Cob$ such that $B(K)=B(L)$. Let $a$ and $b$
denote numbers of circular components of $K$ and $L$, respectively. Since the non-zero matrices $p^{a}\cdot A(K)$ and $p^{b}\cdot A(L)$ are equal, they have the same order $p^{m}\times p^{n}$ and the corresponding elements, which are either zero or the powers of $p$, are identical. We have $a=b$, because $p\geq 2$, and $$A(K)[i,j]=1\ \mbox{iff}\ A(L)[i,j]=1.$$ This means that for every $g_{j}:n\rightarrow p$ and every $f_{i}:m\rightarrow p$ we have $$[g_{j},f_{i}]\in\mathcal{F}^{\ \! =}(R_{K})\ \mbox{iff}\ [g_{j},f_{i}]\in\mathcal{F}^{\
\! =}(R_{L}).$$ As every function $f:(n\times\{0\})\cup(m\times\{1\})\rightarrow p$ is of the form $[g_{j},f_{i}]$ for some $g_{j}:n\rightarrow p$ and
some $f_{i}:m\rightarrow p$, it follows that $\mathcal{F}^{\
\!=}(R_{K})=\mathcal{F}^{\ \!=}(R_{L})$. Applying Proposition \ref{KONDOM}, we conclude that $R_{K}=R_{L}$. By Proposition \ref{ekvkobordizmi}
it follows that $K$ and $L$ are equivalent.
\end{proof}
It is worth pointing out that the faithfulness of the Brauerian functor $B$ is a consequence of the maximality given in \cite[Section 14]{DP4}.

\section{Strict $1$-dimensional Topological Quantum Field Theories}

Let us denote by $+$ a $0$-dimensional manifold, which consists of one point with positive orientation, and by $-$ the same manifold with the opposite orientation. Then, every object of $1Cob$ can be regarded as a sequence
$a_{1}\ldots a_{n}$, where $a_{i} \in \{+,-\}$. Any compact oriented $1$-dimensional manifold $M$ is homeomorphic to the disjoint union of its connected components. Those connected components are either homeomorphic to a closed interval $[0,1]$ or to a circle $S^{1}$. Depending on how the boundary $\partial M$ is decomposed into ingoing and outgoing pieces, we have the following connected cobordisms:

\vspace{0.5em} {\boldmath $1^{0}$} Unit interval $[0,1]$ with its standard orientation regarded as a cobordism from $+$ to
$+$. It represents the identity morphism $id_{+}$ in $1Cob$.

{\boldmath $2^{0}$} Unit interval $[0,1]$ with its standard orientation regarded as a cobordism from $-$ to $-$. It represents the identity morphism $id_{-}$ in $1Cob$.

{\boldmath $3^{0}$} Unit interval $[0,1]$ with its standard orientation regarded as a cobordism from $+-$ to $\emptyset$.
It represents the morphism of $1Cob$, which we denote by $\mathcal{B}$.

{\boldmath $4^{0}$} The cobordism from $-+$ to $\emptyset$, which we denote by $\overline{\mathcal{B}}$.

{\boldmath $5^{0}$} Unit interval $[0,1]$ with its standard orientation regarded as a cobordism from $\emptyset$ to $-+$.
It represents the morphism of $1Cob$, which we denote by $\mathcal{C}$.

{\boldmath $6^{0}$} The cobordism from $\emptyset$ to $+-$, which we denote by $\overline{\mathcal{C}}$.

{\boldmath $7^{0}$} The circle $S^{1}$ regarded as a cobordism from the empty set to itself.

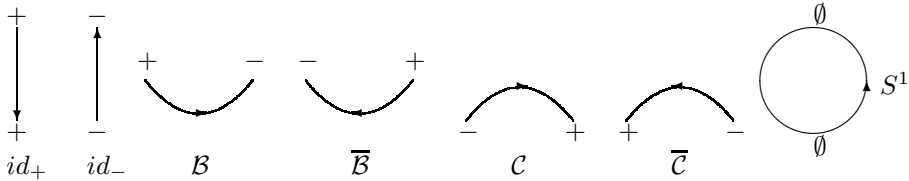
\begin{center}
\begin{picture}(330,60)

\put(2,50){\vector(0,-1){35}} \put(-2,8){$+$} \put(-2,52){$+$}
\put(-2,-5){$id_{+}$}

\put(32,15){\vector(0,1){35}} \put(28,8){$-$} \put(28,52){$-$}
\put(28,-5){$id_{-}$}

\qbezier(50,30)(70,5)(90,30) \put(70,18){\vector(1,0){3}}
\put(47,35){$+$} \put(87,35){$-$}
\put(67,-5){$\mathcal{B}$}

\qbezier(110,30)(130,5)(150,30) \put(130,18){\vector(-1,0){3}}
\put(107,35){$-$} \put(147,35){$+$}
\put(127,-5){$\overline{\mathcal{B}}$}

\qbezier(170,15)(190,40)(210,15) \put(190,28){\vector(1,0){3}}
\put(167,8){$-$} \put(207,8){$+$}
\put(187,-5){$\mathcal{C}$}

\qbezier(230,15)(250,40)(270,15) \put(250,28){\vector(-1,0){3}}
\put(227,8){$+$} \put(267,8){$-$}
\put(247,-5){$\overline{\mathcal{C}}$}

\put(300,30){\circle{40}} \put(300,52){$\emptyset$}
\put(300,2){$\emptyset$} \put(320,30){\vector(0,0){0}}
\put(325,25){$S^{1}$}

\end{picture}
\end{center}

\begin{defn}
A strict $1$-dimensional topological quantum field theory ($1$-TQFT) is a strict symmetric monoidal functor between the category $1Cob$ and the category $Mat_{\mathbb{K}}$.
\end{defn}

The following proposition is motivated by \cite[Proposition~1.1.8]{L1}.

\begin{prop}\label{TQFT}{
Let $F$ be a strict $1$-TQFT. If $p=F(+)$ and $q=F(-)$, then $p=q$.
}
\end{prop}

\begin{proof} Applying the functor $F$ to the morphism $\mathcal{B}: +-
\rightarrow \emptyset$ of $1Cob$, we obtain the morphism $F(\mathcal{B}): F(+-)\rightarrow F(\emptyset)$ of
$Mat_{\mathbb{K}}$. Since the functor $F$ is a strict monoidal, we have $F(+-)=F(+) \otimes F(-)=p\cdot q$ and $F(\emptyset)=1$. Hence, the matrix $F(\mathcal{B})$ has order $1 \times (pq)$. Let us introduce the notation $$F(\mathcal{B})= \left[\!\!
\begin{array}{ccccccccccccccccccc} \beta_{11} & \ldots & \beta_{1q} & \big{|} & \beta_{21} & \ldots & \beta_{2q} & \big{|} & \ldots & \big{|} & \beta_{p1} & \ldots & \beta_{pq}\\
\end{array} \!\!\right]
$$ and $X=\left[\!\!
\begin{array}{cccccccccccccccc} \beta_{11}  & \ldots & \beta_{1q}\\
\beta_{21} & \ldots & \beta_{2q}\\
\vdots & \vdots & \vdots \\
\beta_{p1} & \ldots & \beta_{pq}
\end{array} \!\!\right].
$ Under the standard isomorphism $H:\mathcal{M}_{p\times
q}\rightarrow \mathcal{M}_{1\times pq}$, we have $H(X)=F(\mathcal{B})$.

Similarly, applying the functor $F$ to the morphism $\mathcal{C}:\emptyset\rightarrow -+$, we obtain the morphism $F(\mathcal{C}):1\rightarrow q \cdot p$ of $Mat_{\mathbb{K}}$, i.e.~the matrix of order $(qp) \times 1$.
Set $$F(\mathcal{C})=\left[\!\!
\begin{array}{ccccccccccccccccccc} \gamma_{11} & \ldots & \gamma_{1p} & \big{|} & \gamma_{21} & \ldots  & \gamma_{2p} & \big{|} & \ldots & \big{|} & \gamma_{q1} & \ldots & \gamma_{qp}\\
\end{array} \!\!\right]^{T}$$ and
$Y=\left[\!\!
\begin{array}{cccccccccccccccc} \gamma_{11}  & \ldots & \gamma_{1p}\\
\gamma_{21} & \ldots & \gamma_{2p}\\
\vdots & \vdots & \vdots \\
\gamma_{q1} & \ldots & \gamma_{qp}
\end{array} \!\!\right].
$ Under the standard isomorphism $L:\mathcal{M}_{q\times
p}\rightarrow \mathcal{M}_{qp\times 1}$, we have $L(Y)=F(\mathcal{C})$.

We consider the identity cobordism $id_{+}$ and its decomposition into two cobordisms $id_{+} \otimes \mathcal{C}$ and $\mathcal{B} \otimes
id_{+}$, as shown in the picture below.

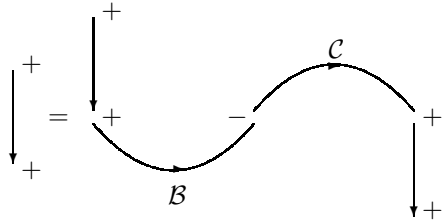
\begin{center}
\begin{picture}(120,85)

\put(-30,60){\vector(0,-1){35}} \put(-27,20){$+$}
\put(-27,60){$+$} \put(-17,40){$=$}

\put(0,80){\vector(0,-1){35}} \put(3,40){$+$} \put(3,80){$+$}

\qbezier(60,45)(90,80)(120,45) \put(89,62){\vector(1,0){3}}
\put(50,40){$-$} \put(123,40){$+$}

\qbezier(0,40)(30,5)(60,40) \put(31,23){\vector(1,0){3}}

\put(120,40){\vector(0,-1){35}} \put(123,5){$+$}

\put(28,10){$\mathcal{B}$} \put(88,65){$\mathcal{C}$}

\end{picture}
\end{center}

\noindent Applying the functor $F$ to both sides of the equation $id_{+}=(\mathcal{B} \otimes id_{+}) \circ (id_{+} \otimes
\mathcal{C})$ and using the functorial properties of $F$, we obtain \beaz id_{F(+)}=(F(\mathcal{B}) \otimes
id_{F(+)}) \cdot(id_{F(+)} \otimes
F(\mathcal{C})), \mbox{i.e.}\,\,\\
E_{p}=(F(\mathcal{B}) \otimes E_{p}) \cdot (E_{p} \otimes
F(\mathcal{C})), \eeaz where $E_{p}$ is the identity matrix of order $p$. As a result of matrix multiplication, we get the system of $p^2$ equations
\begin{equation}\label{sistem}\sum_{k=1}^{q} \beta_{ik}\cdot \gamma_{kj} =
\delta_{ij},\;\; i,j\in \{1,\ldots,p\}\end{equation} that can be written in the matrix form $$X \cdot Y= E_{p}.$$
Therefore, the matrix $X\in \mathcal{M}_{p\times q}$ has a right inverse, so its rows are linearly independent, i.e.~row rank is equal to $p$.

If we apply the functor $F$ to
 the morphism $\tau_{-+}: -+
\rightarrow +-$ of $1Cob$, shown in the following picture
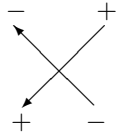
\begin{center}
\begin{picture}(120,60)

\put(53,17){\vector(-1,1){30}} \put(27,17){\line(1,1){30}}
\put(29,19){\vector(-1,-1){3}} \put(20,50){$-$}
\put(54,50){$+$} \put(22,8){$+$} \put(50,8){$-$}

\end{picture}
\end{center}
we obtain the morphism $F(\tau_{-+}): F(-) \otimes F(+)\rightarrow
F(+) \otimes F(-)$ of $Mat_{\mathbb{K}}$, i.e.~the matrix of order $(pq) \times (qp)$.
The symmetry of $F$ implies $F(\tau_{-+})=S_{q,p}$, where $S_{q,p}$ is the commutation matrix.
The crucial fact is that the commutative matrix $S_{q,p}$ satisfies the following conditions:
\begin{equation}
\label{Eq_11}S_{qp} \cdot L(Y)=L(Y^{T})
\end{equation}
\begin{equation}
\label{Eq_12}H(X) \cdot S_{qp}=H(X^{T})
\end{equation}
(more details can be found in \cite{MN,
crnac, V})

Since the cobordisms $\overline{\mathcal{B}}$ and $\overline{\mathcal{C}}$ can be decomposed as
$\overline{\mathcal{B}}=\mathcal{B} \circ \tau_{-+}$ and
$\overline{\mathcal{C}}=\tau_{-+} \circ \mathcal{C}$,

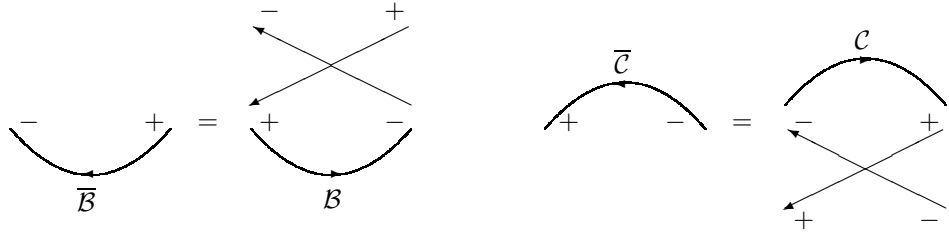
\begin{center}
\begin{picture}(160,90)

\qbezier(-90,40)(-60,5)(-30,40) \put(-60,23){\vector(-1,0){3}}
\put(-87,40){$-$} \put(-40,40){$+$}

\put(-20,40){$=$}

\put(60,49){\vector(-2,1){59}} \put(0,49){\line(2,1){59}}
\put(2,50){\vector(-2,-1){3}} \put(3,82){$-$} \put(50,82){$+$}
\put(3,40){$+$} \put(50,40){$-$}

\qbezier(0,40)(30,5)(60,40) \put(31,23){\vector(1,0){3}}

\qbezier(110,40)(140,75)(170,40) \put(139,57){\vector(-1,0){3}}
\put(115,40){$+$} \put(155,40){$-$} \put(180,40){$=$}

\put(260,10){\vector(-2,1){59}} \put(200,10){\line(2,1){59}}
\put(202,11){\vector(-2,-1){3}} \put(203,40){$-$}
\put(250,40){$+$} \put(203,3){$+$} \put(250,3){$-$}

\qbezier(200,49)(230,84)(260,49) \put(229,66){\vector(1,0){3}}
\put(-65,9){$\overline{\mathcal{B}}$}
\put(27,10){$\mathcal{B}$}
\put(136,61){$\overline{\mathcal{C}}$}
\put(226,70){$\mathcal{C}$}

\end{picture}
\end{center}

it follows that $F(\overline{\mathcal{B}})=F(\mathcal{B})\cdot
F(\tau_{-+})$ and $F(\overline{\mathcal{C}})=F(\tau_{-+})\cdot
F(\mathcal{C})$, i.e.~ \beaz
F(\overline{\mathcal{B}})=\left[\!\!
\begin{array}{cccccccccccccccc} \beta_{11} & \ldots & \beta_{1q} & \big{|} &\beta_{21} & \ldots & \beta_{2q} & \big{|} &\ldots & \big{|} &\beta_{p1} & \ldots & \beta_{pq}\\
\end{array} \!\!\right]\cdot S_{qp}\stackrel{\mathrm{(\ref{Eq_12})}}{=\joinrel=\joinrel=}\,\,\\
\left[\!\!
\begin{array}{cccccccccccccccc} \beta_{11} & \ldots & \beta_{p1} & \big{|} &\beta_{12} & \ldots & \beta_{p2} & \big{|} &\ldots & \big{|} &\beta_{1q} & \ldots & \beta_{pq}\\
\end{array} \!\!\right]\eeaz
\beaz F(\overline{\mathcal{C}})=S_{qp}\cdot \left[\!\!
\begin{array}{cccccccccccccccc} \gamma_{11} & \ldots & \gamma_{1p} & \big{|} &\gamma_{21} & \ldots & \gamma_{2p} & \big{|} & \ldots & \big{|} &\gamma_{q1} & \ldots & \gamma_{qp}\\
\end{array} \!\!\right]^{T}\stackrel{\mathrm{(\ref{Eq_11})}}{=\joinrel=\joinrel=}\,\,\\
\left[\!\!
\begin{array}{cccccccccccccccc} \gamma_{11} & \ldots & \gamma_{q1} & \big{|} & \gamma_{12} & \ldots & \gamma_{q2} & \big{|} & \ldots & \big{|} &\gamma_{1p} & \ldots & \gamma_{qp}\\
\end{array} \!\!\right]^{T}\eeaz

Analogously, we can decompose the cobordism $id_{-}$ into two cobordisms $id_{-} \otimes
\overline{\mathcal{C}}$ and $\overline{\mathcal{B}} \otimes
id_{-}$,
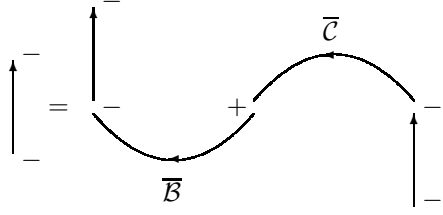
\begin{center}
\begin{picture}(120,100)

\put(-30,45){\vector(0,1){35}} \put(-27,40){$-$}
\put(-27,80){$-$} \put(-17,60){$=$}

\put(0,65){\vector(0,1){35}} \put(3,60){$-$} \put(3,100){$-$}

\qbezier(60,65)(90,100)(120,65) \put(89,82){\vector(-1,0){3}}
\put(50,60){$+$} \put(123,60){$-$}

\qbezier(0,60)(30,25)(60,60) \put(31,43){\vector(-1,0){3}}

\put(120,25){\vector(0,1){35}} \put(123,25){$-$}

\put(26,27){$\overline{\mathcal{B}}$}
\put(86,88){$\overline{\mathcal{C}}$}

\end{picture}
\end{center}
\noindent Applying $F$, we see that \beaz
id_{F(-)}=(F(\overline{\mathcal{B}}) \otimes id_{F(-)})
\cdot(id_{F(-)} \otimes
F(\overline{\mathcal{C}})), \mbox{i.e.}\,\,\\
E_{q}=(F(\overline{\mathcal{B}}) \otimes E_{q}) \cdot (E_{q}
\otimes F(\overline{\mathcal{C}})), \eeaz where $E_{q}$ is the identity matrix of order $q$. This yields the system of $q^{2}$ equations
$$\sum_{i=1}^{p} \beta_{ik}\cdot \gamma_{li} =
\delta_{kl},\;\; k,l\in \{1,\ldots,q\}$$ equivalent to the matrix equation
$$Y \cdot X= E_{q}.$$
Consequently, the matrix $X\in \mathcal{M}_{p\times q}$ has a left inverse, so its columns are linearly independent, i.e.~its column rank is equal to $q$. Since the row rank and the column rank of a matrix are equal, it follows that $p=q$ as claimed. Note that we have proved more, namely that the matrices $X$ and $Y$ are inverses to each other.
\end{proof}

Every strict $1$-TQFT, $F:1Cob \rightarrow Mat_{\mathbb{K}}$ is completely determined on objects by the image of positively oriented point $p=F(+)$. By the Proposition \ref{TQFT} the value of $F$ on any object $a_{1}
\ldots a_{n}$ of $1Cob$, consisting of $k$
positively and $l$ negatively oriented points, is $p^{k}\cdot p^{l}=p^{k+l}=p^{n}$. Hence, every strict $1$-TQFT, mapping the manifold consisting of one point to a number $p\geq 2$, on objects coincide with Brauerian representation. Since every cobordism is a finite tensor product of connected cobordisms composed with some applications of $\tau$'s, we only need to know where $F$ sends $id_{+}$, $id_{-}$, $\mathcal{B}$,
$\overline{\mathcal{B}}$, $\mathcal{C}$,
$\overline{\mathcal{C}}$ and $S^{1}$. The matrices $F(id_{+})$ and
$F(id_{-})$ are equal to the identity matrix $E_{p}$ of order $p$.
Due to the fact that the cobordisms $\overline{\mathcal{B}}$ and
$\overline{\mathcal{C}}$ can be decomposed into compositions of the cobordism $\tau_{-+}$ with $\mathcal{B}$ and $\mathcal{C}$, respectively, and the equality $F(\tau_{-+})=S_{pp}$, we only need to describe the following matrices $$F(\mathcal{B})= \left[\!\!
\begin{array}{cccccccccccccccc} \beta_{11} & \ldots & \beta_{1p} &\big{|} & \beta_{21} & \ldots & \beta_{2p} &\big{|} & \ldots &\big{|} & \beta_{p1} & \ldots & \beta_{pp}\\
\end{array} \!\!\right]\in \mathcal{M}_{1\times p^{2}}
,$$ $$F(\mathcal{C})=\left[\!\!
\begin{array}{cccccccccccccccc} \gamma_{11} & \ldots & \gamma_{1p} &\big{|} & \gamma_{21} & \ldots & \gamma_{2p} &\big{|} & \ldots &\big{|} & \gamma_{p1} & \ldots & \gamma_{pp}\\
\end{array} \!\!\right]^{T}\in \mathcal{M}_{p^{2}\times 1}$$ and
$F(S^{1})\in \mathcal{M}_{1\times 1}$.

\begin{prop}\label{TQFT2}{
$F(S^{1})=p$.}
\end{prop}
\begin{proof} Let us look at $S^{1}$ as the composition $\mathcal{B} \circ
\tau_{-,+} \circ \mathcal{C}$.

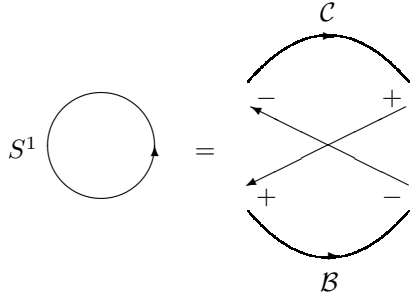
\bigskip

\begin{center}
\begin{picture}(160,110)
\put(45,65){\circle{40}} \put(65,65){\vector(0,0){0}}

\put(80,60){$=$}

\put(160,50){\vector(-2,1){59}} \put(100,50){\line(2,1){59}}
\put(102,51){\vector(-2,-1){3}} \put(103,80){$-$}
\put(150,80){$+$} \put(103,43){$+$} \put(150,43){$-$}

\qbezier(100,89)(130,124)(160,89)
\put(129,106){\vector(1,0){3}}

\qbezier(100,40)(130,5)(160,40) \put(131,23){\vector(1,0){3}}

\put(127,10){$\mathcal{B}$} \put(127,112){$\mathcal{C}$}
\put(10,60){$S^{1}$}

\end{picture}
\end{center} Thus we see that $F(S^{1})=F(\mathcal{B})\cdot F(\tau_{-+}) \cdot
F(\mathcal{C})$, i.e.~\beaz F(S^{1})=\left[\!\!
\begin{array}{cccccccccccccccc} \beta_{11} & \ldots & \beta_{1p} &\big{|} & \beta_{21} & \ldots & \beta_{2p} &\big{|} & \ldots &\big{|} & \beta_{p1} & \ldots & \beta_{pp}\\
\end{array} \!\!\right] \cdot S_{pp} \cdot\,\,\\
\left[\!\!
\begin{array}{cccccccccccccccc} \gamma_{11} & \ldots & \gamma_{1p} &\big{|} & \gamma_{21} & \ldots & \gamma_{2p} &\big{|} & \ldots &\big{|} & \gamma_{p1} & \ldots & \gamma_{pp}\\
\end{array} \!\!\right]^{T}\stackrel{\mathrm{(\ref{Eq_11})}}{=\joinrel=\joinrel=}\,\,\\
\left[\!\!
\begin{array}{cccccccccccccccc} \beta_{11} & \ldots & \beta_{1p} &\big{|} & \beta_{21} & \ldots & \beta_{2p} &\big{|} & \ldots &\big{|} & \beta_{p1} & \ldots & \beta_{pp}\\
\end{array} \!\!\right] \cdot \,\,\\
\left[\!\!
\begin{array}{cccccccccccccccc} \gamma_{11} & \ldots & \gamma_{p1} &\big{|} & \gamma_{12} & \ldots & \gamma_{p2} &\big{|} & \ldots &\big{|} & \gamma_{1p} & \ldots & \gamma_{pp}\\
\end{array} \!\!\right]^{T}=\,\,\\
\underbrace{\beta_{11}\gamma_{11}+\beta_{12}\gamma_{21}+\ldots+\beta_{1p}\gamma_{p1}}_{1}+\underbrace{\beta_{21}\gamma_{12}+\beta_{22}\gamma_{22}+\ldots+\beta_{2p}\gamma_{p2}}_{1}+
\ldots\,\,\\
+\underbrace{\beta_{p1}\gamma_{1p}+\beta_{p2}\gamma_{2p}+\ldots+\beta_{pp}\gamma_{pp}}_{1}\stackrel{\mathrm{(\ref{sistem})}}{=\joinrel=\joinrel=}p\eeaz

\end{proof}

We can rewrite the matrix equation $X \cdot Y=E_{p}$ (see for instance \cite[Section 2.8]{J1}, \cite{V}) as $H(X)\cdot(E_{p} \otimes Y)=H(E_{p}), \mbox{i.e.}$~\bea
\label{Eq beta} \left[\!\!
\begin{array}{cccccccccccccccc} \beta_{11} & \ldots & \beta_{1p} &\big{|} & \beta_{21} & \ldots & \beta_{2p} & \big{|} & \ldots &\big{|} & \beta_{p1} & \ldots & \beta_{pp}\\
\end{array} \!\!\right]\cdot (E_{p} \otimes Y)= \,\, \nonumber \\
\left[\!\!
\begin{array}{cccccccccccccccc} 1 & 0 & \ldots & 0 &\big{|} & 0 & 1 &\ldots & 0 &\big{|} & \ldots & \big{|} & 0 & 0 & \ldots & 1\\
\end{array} \!\!\right],
\eea as well as $(X \otimes E_{p})\cdot L(Y)=L(E_{p}),
\mbox{i.e.}$~\bea
\label{Eqgama1} (X \otimes
E_{p})\cdot\left[\!\!
\begin{array}{cccccccccccccccc} \gamma_{11} & \ldots & \gamma_{1p} &\big{|} & \gamma_{21} & \ldots & \gamma_{2p} &\big{|} & \ldots &\big{|} & \gamma_{p1} & \ldots & \gamma_{pp}\\
\end{array} \!\!\right]^{T}= \,\, \nonumber \\
\left[\!\!
\begin{array}{cccccccccccccccc} 1 & 0 & \ldots & 0 & \big{|} & 0 & 1 &\ldots & 0 & \big{|} & \ldots & \big{|} & 0 & 0 & \ldots & 1\\
\end{array} \!\!\right]^{T}.
\eea Since the Brauerian representation assigns the matrices $\hfill\hfill\hfill\hfill\hfill\hfill\hfill\hfill\hfill\hfill\hfill\hfill\hfill\hfill\hfill$
$\left[\!\!
\begin{array}{cccccccccccccccc} 1 & 0 & \ldots & 0 &\big{|} & 0 & 1 &\ldots & 0 &\big{|} & \ldots & \big{|}& 0 & 0 & \ldots & 1\\
\end{array} \!\!\right]$
and\\ $\left[\!\!
\begin{array}{cccccccccccccccc} 1 & 0 & \ldots & 0 & \big{|} & 0 & 1 &\ldots & 0 & \big{|} & \ldots & \big{|} & 0 & 0 & \ldots & 1\\
\end{array} \!\!\right]^{T}$ to the cobordisms $\mathcal{B}$ and $\mathcal{C}$, respectively, we conclude that every strict $1$-TQFT on $\mathcal{B}$ and $\mathcal{C}$ coincides with the Brauerian representation up to multiplication by invertible matrices.

\begin{prop}\label{TQFT3}{
Let $F:1Cob \rightarrow Mat_{\mathbb{K}}$ be a strict $1$-TQFT such that $F(+)=p\geq 2$, and $B:1Cob \rightarrow Mat_{\mathbb{K}}$ be the Brauerian representation. Then, there is a monoidal natural isomorphism $\theta: B \Rightarrow F$.}
\end{prop}
\begin{proof} Let us assign to each object $a$ of $1Cob$ an invertible morphism $\theta_{a}:B(a)\rightarrow F(a)$ of
$Mat_{\mathbb{K}}$, i.e.~an invertible matrix in the following way. We first define
$\theta_{\emptyset}:1\rightarrow 1$, $\theta_{+}:p \rightarrow
p$ and $\theta_{-}:p \rightarrow p$ to be $E_{1}$, $E_{p}$ and
$X^{-1}$, respectively. Then, for every object $a=a_{1}\ldots a_{n}$ of $1Cob$ we define $\theta_{a}$ by the Kronecker product $\theta_{a_{1}} \otimes \ldots \otimes
\theta_{a_{n}}$. We proceed to show that for every morphism $f:a\rightarrow a'$ of $1Cob$ the following diagram commutes in $Mat_{\mathbb{K}}$
\begin{center}
\begin{picture}(120,60)(0,-5)

\put(0,40){\makebox(0,0){$B(a)$}}
\put(0,0){\makebox(0,0){$B(a')$}}
\put(100,40){\makebox(0,0){$F(a)$}}
\put(100,0){\makebox(0,0){$F(a')$}}

\put(50,47){\makebox(0,0){$\theta_{a}$}}
\put(50,-7){\makebox(0,0){$\theta_{a'}$}}
\put(-20,20){\makebox(0,0){$B(f)$}}
\put(118,20){\makebox(0,0){$F(f)$}}

\put(20,40){\vector(1,0){60}} \put(20,0){\vector(1,0){60}}
\put(0,30){\vector(0,-1){20}} \put(100,30){\vector(0,-1){20}}

\end{picture}
\end{center}

It suffices to prove for generators $id_{+}$, $id_{-}$,
$S^{1}$, $\mathcal{B}$, $\mathcal{C}$, $\overline{\mathcal{B}}$
and $\overline{\mathcal{C}}$. From $B(id_{+})=F(id_{+})=E_p$ and
$B(id_{-})=F(id_{-})=E_p$, we obtain $\theta_+ \cdot E_p=E_p \cdot
\theta_+$ and $\theta_- \cdot E_p=E_p \cdot \theta_-$. We have
$\theta_\emptyset \cdot B(S^{1})=F(S^{1}) \cdot
\theta_\emptyset $ because $1 \cdot p=p \cdot 1$. The following diagrams commute
\begin{center}
\begin{picture}(370,60)(0,-5)

\put(0,50){\makebox(0,0){$B(+-)$}}
\put(0,10){\makebox(0,0){$B(\emptyset)$}}
\put(100,50){\makebox(0,0){$F(+-)$}}
\put(100,10){\makebox(0,0){$F(\emptyset)$}}

\put(53,57){\makebox(0,0){$\theta_{+-}$}}
\put(-22,30){\makebox(0,0){$B(\stackrel{+\ \!-}{\cup})$}}
\put(122,30){\makebox(0,0){$F(\stackrel{+\ \!-}{\cup})$}}

\put(30,50){\vector(1,0){40}} \put(50,10){\makebox(0,0){$=$}}
\put(0,40){\vector(0,-1){20}} \put(100,40){\vector(0,-1){20}}

\put(220,50){\makebox(0,0){$B(\emptyset)$}}
\put(220,10){\makebox(0,0){$B(-+)$}}
\put(320,50){\makebox(0,0){$F(\emptyset)$}}
\put(320,10){\makebox(0,0){$F(-+)$}}

\put(273,0){\makebox(0,0){$\theta_{-+}$}}
\put(198,30){\makebox(0,0){$B(\underset{-+}{\cap})$}}
\put(342,30){\makebox(0,0){$F(\underset{-+}{\cap})$}}

\put(270,50){\makebox(0,0){$=$}} \put(250,10){\vector(1,0){40}}
\put(220,40){\vector(0,-1){20}} \put(320,40){\vector(0,-1){20}}

\end{picture}
\end{center}
\noindent which follows from $B(\stackrel{+\
\!-}{\cup})\stackrel{\mathrm{(\ref{Eq
beta})}}{=\joinrel=\joinrel=}F(\stackrel{+\ \!-}{\cup})\cdot
(E_{p}\otimes X^{-1})=F(\stackrel{+\ \!-}{\cup})\cdot
(\theta_+\otimes \theta_-)=F(\stackrel{+\
\!-}{\cup})\cdot\theta_{+-}$ and

$$F(\underset{-+}{\cap})\stackrel{\mathrm{(\ref{Eqgama1})}}{=\joinrel=\joinrel=}(X^{-1} \otimes E_{p})\cdot
B(\underset{-+}{\cap})=(\theta_- \otimes \theta_+)\cdot
B(\underset{-+}{\cap})=\theta_{-+}\cdot
B(\underset{-+}{\cap}).$$
Using the key property of the commutation matrix that enables us to interchange the two matrices of a Kronecker product,
we can see that \beaz B(\stackrel{-\
\!+}{\cup})=B(\stackrel{+\ \!-}{\cup})\cdot
S_{pp}=F(\stackrel{+\ \!-}{\cup})\cdot (E_{p} \otimes
X^{-1})\cdot S_{pp}=\,\,\\F(\stackrel{+\ \!-}{\cup})\cdot
S_{pp}\cdot (X^{-1}\otimes E_{p})= F(\stackrel{-\
\!+}{\cup})\cdot (X^{-1}\otimes E_{p})=F(\stackrel{-\
\!+}{\cup})\cdot \theta_{-+}\eeaz \beaz
F(\underset{+-}{\cap})=S_{pp}\cdot F(\underset{-+}{\cap})
=S_{pp}\cdot (X^{-1} \otimes E_{p})\cdot
B(\underset{-+}{\cap})=\,\,\\=(E_{p} \otimes X^{-1})\cdot
S_{pp} \cdot B(\underset{-+}{\cap})=\theta_{+-}\cdot
B(\underset{+-}{\cap}).\eeaz

\end{proof}

A direct consequence of this last result and the faithfulness of the Brauerian representation is as follows.
\begin{cor}\label{TQFT4}{
Every strict $1$-TQFT, $F:1Cob \rightarrow Mat_{\mathbb{K}}$, such that $F(+)=p\geq 2$, is faithful.}
\end{cor}

\section{Strong $1$-dimensional Topological Quantum Field Theories}

The category $Vect_{\mathbb{K}}$ of finite dimensional vector spaces over a fixed field $\mathbb{K}$ with ordinary tensor product $\otimes$ and $1$-dimensional vector space $\mathbb{K}$ as the unit is a symmetric monoidal, but not a strict monoidal. By the universal property of the tensor product, there is a unique isomorphism
$\alpha_{V,W,U}: V \otimes (W \otimes U) \cong (V \otimes W)
\otimes U$, such that $v\otimes (w \otimes u)\mapsto (v \otimes
w)\otimes u$. The structural isomorphism $\lambda_{V}: \mathbb{K}
\otimes V \cong V$ and $\rho_{V}: V \otimes \mathbb{K} \cong V$
are given by $a \otimes v\mapsto a v$ and $v \otimes a\mapsto
a v$, respectively. The symmetry is brought by $\sigma_{V,W}:V \otimes W \cong W \otimes V$
defined by $v \otimes  w\rightarrow w \otimes v$.

\begin{defn} Strong $1$-TQFT is a strong symmetric monoidal functor $\hfill\hfill\hfill\hfill\hfill\hfill$ $(F,F_{0},F_{2})$ between the strict symmetric monoidal category $(1Cob, \otimes, \emptyset, \tau_{n,m})$ and the non-strict symmetric monoidal category $(Vect_{\mathbb{K}}, \otimes, \mathbb{K},
\alpha, \lambda, \rho, \sigma_{V,W})$ (for the notion of strong monoidal functor see \cite{ML1}).
\end{defn}

For every closed oriented null-dimensional manifold, regarded as a sequence of points $a=a_{1} \ldots a_{n}$, where $a_{i} \in \{+,-\}$, the functor $F$ assigns a vector space $F(a)$, and for every oriented $1$-cobordism $K$ from
$a$ to $b$ it assigns a linear map $F(K):F(a) \rightarrow F(b)$.
The components of a natural transformation $F_{2}$ are isomorphisms $F_{2}(a,b):
F(a) \otimes F(b) \xrightarrow{\cong} F(ab)$, and $F_{0}: \mathbb{K} \xrightarrow{\cong}
F(\emptyset)$ is also an isomorphism of $Vect_{\mathbb{K}}$. Together, they must make the diagrams involving the structural maps $\alpha$, $\lambda$, $\rho$, and $\sigma$ commute in $Vect_{\mathbb{K}}$ (see \cite{ML1}).

Our goal is to prove that every strong $1$-TQFT, $F:1Cob \rightarrow
Vect_{\mathbb{K}}$, mapping $a_{i} \in \{+,-\}$ to a vector space of dimension at least $2$, is faithful.
For this purpose, let us consider the category $VectB_{\mathbb{K}}$ whose objects are the ordered pairs $(V,e)$, where $V$ is a finite dimensional vector space, and $e$ is a chosen ordered basis for $V$. The morphisms $(V,e)\rightarrow(V',e')$ are just the usual linear maps $V\rightarrow V'$. If $e=[e_{0}, \ldots, e_{n-1}]$ and $f=[f_{0}, \ldots, f_{m-1}]$ are ordered bases for $V$ and $W$, respectively, let $e \otimes f$ denote the ordered basis for $V \otimes W$, having vector $\varphi(e_{i},f_{j})$ at the $[i\cdot m+ j]$-th position, where $\varphi:V
\times W\rightarrow V \otimes W$ is canonical bilinear map. The monoidal structure of $VectB_{\mathbb{K}}$ is given on objects by $$(V,e) \otimes (W,f)=(V \otimes W, e \otimes f),$$ with $(\mathbb{K},1_{\mathbb{K}})$ serving as the unit. The tensor product of two morphisms is defined in the same way as in $Vect_{\mathbb{K}}$.

We now proceed to introduce a new functor $F^{*}: 1Cob \rightarrow
VectB_{\mathbb{K}}$. It is recursively defined on objects in the following way, while it coincides with $F$ on morphisms.

\noindent
The image of $1_{\mathbb{K}}$ under the isomorphism $F_{0}:\mathbb{K} \xrightarrow{\cong}
F(\emptyset)$ is taken to be the basis of $F(\emptyset)$, so we set $F^{*}(\emptyset)=(F(\emptyset), F_{0}(1_{\mathbb{K}}))$.

Once we have chosen bases $e_{+}$ and $e_{-}$ for the spaces $F(+)$ and $F(-)$, respectively, we define $F^{*}(+)=(F(+),e_{+})$ and
$F^{*}(-)=(F(-),e_{-})$.
If $a=a_{1} \ldots a_{n}$ is an object of length $n$, we define $F^{*}(a_{1} \ldots a_{n})=(F(a_{1} \ldots a_{n}), e_{a_{1}
\ldots a_{n}})$, where a basis $e_{a_{1} \ldots a_{n}}$ for the space $F(a_{1} \ldots a_{n})$ is obtained by taking the image of an ordered
basis of domain under the isomorphism $$F_{2}(a_{1}, a_{2} \ldots a_{n}): F(a_{1}) \otimes
F(a_{2}\ldots a_{n})\rightarrow F(a_{1} \ldots a_{n}).$$ More precisely, if
$e_{a_{1}}$ is a given basis for $F(a_{1})$, and $e_{a_{2} \ldots a_{n}}$ is recursively defined basis for $F(a_{2} \ldots a_{n})$, then
the basis for the space $F(a_{1} \ldots a_{n})$ is taken to be the image of basis $e_{a_{1}} \otimes
e_{a_{2} \ldots a_{n}}$ under the isomorphism $F_{2}$.

Our next task is to prove that $F^{*}$ is faithful.
To do this, take a dimensional functor $G:VectB_{\mathbb{K}} \rightarrow
Mat_{\mathbb{K}}$, sending a vector space to its dimension, and a linear map to its matrix with respect to the chosen ordered bases
$$G(V,e)=\dim V,$$
$$G(L:(V,e)\rightarrow (V',e'))= [L]_{e,e'}.$$

\begin{lem}\label{Galpha}
{Let $e=[e_{0}, \ldots, e_{n-1}]$, $f=[f_{0}, \ldots,
f_{m-1}]$ and $g=[g_{0}, \ldots, g_{k-1}]$ be bases for $U$, $V$ and $W$, respectively.
If $\alpha:(U,e)\otimes ((V,f)\otimes
(W,g))\rightarrow ((U,e)\otimes (V,f))\otimes (W,g)$ is defined by
$$\alpha(e_{i} \otimes (f_{j} \otimes g_{h}))=(e_{i} \otimes f_{j}) \otimes g_{h},$$
then
$$G(\alpha)=E_{n\cdot m
\cdot k},$$ where $E_{n\cdot m \cdot k}$ is the identity matrix of order $n\cdot m \cdot k$.}
\end{lem}
\begin{proof} The vector $e_{i}\otimes (f_{j} \otimes g_{h})$ is the $[i \cdot(mk)+j
\cdot k +h]$-th element of the basis $e\otimes (f \otimes g)$ for the space $U \otimes (V \otimes W)$, and
the vector $(e_{i} \otimes f_{j}) \otimes g_{h}$ is the $[(i \cdot m +j)\cdot k+h]$-th element of
the basis $(e \otimes f) \otimes g$ for the space $(U \otimes V) \otimes W$. The image under $\alpha$ of the $l$-th basis vector
$e_{i}\otimes (f_{j} \otimes g_{h})$ is the $l$-th basis vector $(e_{i} \otimes f_{j}) \otimes g_{h}$. Thus,
the matrix representation of $\alpha$ with respect to bases $e\otimes (f \otimes g)$ and $(e \otimes f) \otimes g$ is the identity matrix
of the appropriate order.
\end{proof}

\begin{lem}\label{jedmat}
{Given any objects $a$ and $b$ of $1Cob$, we have
$$G(F_{2}(a,b): F(a)\otimes F(b) \xrightarrow{\cong} F(ab))=E.$$ }
\end{lem}
\begin{proof} The proof is by induction on the length of the object $a$.
Base case:
Let $a_{1}\in \{+,-\}$ be an object of length $1$ and let $b=a_{2}\ldots a_{n}$ be an object of an arbitrary length.
Fix ordered basis $e_{a_{1}}=[e_{0},\ldots, e_{p-1}]$, $p\geq 2$, and $e_{a_{2}\ldots a_{n}}=[f_{0}, \ldots, f_{m-1}]$ for $F(a_{1})$ and
$F(a_{2}\ldots a_{n})$, respectively.
Then the corresponding basis $[g_{0},\ldots, g_{pm-1}] $ for $F(a_{1})\otimes
F(a_{2}\ldots a_{n})$ is given by $g_{i\cdot m+j}=e_{i}\otimes f_{j}$.
The matrix representation of $$F_{2}(a_{1}, a_{2}\ldots a_{n}):F(a_{1}) \otimes F(a_{2}\ldots a_{n})\xrightarrow{\cong}
F(a_{1}\ldots a_{n})$$
with respect to the bases $[g_{0},\ldots, g_{pm-1}]$ and $[F_{2}(g_{0}),\ldots, F_{2}(g_{pm-1})]$
is the identity matrix of order $p\cdot m$, i.e. $G(F_{2}(a_{1},b))=E$.
Induction step: Suppose that the claim is true for all objects $a$ of length less than $n$, where $n>1$.
Take an arbitrary object $a=a_{1}\ldots a_{n}$ of $1Cob$.
By the commutativity of the following diagram

\begin{center}
\begin{picture}(120,100)(0,-5)

\put(-10,40){\makebox(0,0){$F(a_{1})\otimes F(a_{2}\ldots a_{n}b)$}}
\put(-5,0){\makebox(0,0){$F(a_{1}(a_{2}\ldots a_{n}b))$}}
\put(130,40){\makebox(0,0){$F(a_{1}a_{2}\ldots a_{n})\otimes F(b)$}}
\put(105,0){\makebox(0,0){$F((a_{1}a_{2}\ldots a_{n})b)$}}
\put(-33,80){\makebox(0,0){$F(a_{1})\otimes (F(a_{2}\ldots a_{n})\otimes F(b))$}}
\put(130,80){\makebox(0,0){$(F(a_{1}) \otimes F(a_{2}\ldots a_{n})) \otimes F(b)$}}

\put(53,87){\makebox(0,0){$\alpha$}}
\put(-40,20){\makebox(0,0){$F_{2}(a_{1}, a_{2}\ldots a_{n}b)$}}
\put(143,20){\makebox(0,0){$F_{2}(a_{1}a_{2}\ldots a_{n}, b)$}}
\put(-43,60){\makebox(0,0){$\id \otimes F_{2}(a_{2}\ldots a_{n},b)$}}
\put(148,60){\makebox(0,0){$F_{2}(a_{1},a_{2}\ldots a_{n}) \otimes \id$}}

\put(35,80){\vector(1,0){30}} \put(50,0){\makebox(0,0){$=$}}
\put(0,30){\vector(0,-1){20}} \put(100,30){\vector(0,-1){20}}
\put(0,70){\vector(0,-1){20}} \put(100,70){\vector(0,-1){20}}

\end{picture}
\end{center}
we have that $$F_{2}(a_{1}a_{2}\ldots a_{n}, b)=$$$$F_{2}(a_{1},a_{2}\ldots a_{n}b)\circ (\id \otimes F_{2}(a_{2}\ldots a_{n}, b)) \circ \alpha^{-1}
\circ (F_{2}(a_{1},a_{2}\ldots a_{n})\otimes \id)^{-1}.$$
Now, by the induction hypothesis and Lemma \ref{Galpha} it follows that $$G(F_{2}(a_{1}a_{2}\ldots a_{n}, b))=E.$$ \end{proof}

\begin{lem}\label{Kroneker}{{(}see \cite[Chapter 11, Proposition 17]{DF1}}{)}\,{
Let $L:(V,e)\rightarrow (V',e')$ and $H:(W,f)\rightarrow
(W',f')$ be linear maps of finite dimensional vector spaces.
Then the Kronecker product of matrices
$[L]_{e,e'}$ and $[H]_{f,f'}$, representing $L$ and $H$,
is equal to the matrix $[L \otimes H]_{e \otimes f, e' \otimes f'}$,
representing $L \otimes H: V
\otimes W\rightarrow V' \otimes W'$, i.e.
$$G(L \otimes H)= G(L) \otimes G(H).$$
}
\end{lem}

\begin{prop}\label{komp}{
The composition $GF^{*}:1Cob \rightarrow Mat_{\mathbb{K}}$ is a strict monoidal functor.}
\end{prop}
\begin{proof}
It can be easily seen that $GF^{*}$ maps unit to unit
$$(GF^{*})(\emptyset)=G(F(\emptyset), F_{0}(1_{\mathbb{K}}))= \dim(F(\emptyset))=1.$$
For any two objects $a$ and $b$ of $1Cob$ we have

\begin{tabbing}
\hspace{1.5em}$(GF^{*})(ab)$ \= $= G(F(ab),e_{ab})=\dim F(ab)= \dim(F(a)\otimes
F(b))$
\\[1ex]
\> $= \dim(F(a)) \cdot \dim(F(b))= G(F(a),e_{a})\cdot
G(F(b),e_{b})$
\\[1ex]
\> $= G(F^{*}(a)) \cdot G(F^{*}(b))$.
\end{tabbing}

By the naturality of $F_{2}$, the following diagram commutes for every two morphisms
$f:a \rightarrow a'$ and $g:b\rightarrow
b'$
\begin{center}
\begin{picture}(120,60)(0,-5)

\put(0,40){\makebox(0,0){$F(a)\otimes F(b)$}}
\put(0,0){\makebox(0,0){$F(a')\otimes F(b')$}}
\put(100,40){\makebox(0,0){$F(ab)$}}
\put(100,0){\makebox(0,0){$F(a'b')$}}

\put(60,47){\makebox(0,0){$F_{2}(a,b)$}}
\put(60,-12){\makebox(0,0){$F_{2}(a',b')$}}
\put(-25,20){\makebox(0,0){$Ff \otimes Fg$}}
\put(125,20){\makebox(0,0){$F(f \otimes g)$}}

\put(40,40){\vector(1,0){35}} \put(40,0){\vector(1,0){35}}
\put(0,30){\vector(0,-1){20}} \put(100,30){\vector(0,-1){20}}

\end{picture}
\end{center}
It follows that
\begin{tabbing}
\hspace{1.5em}$(GF^{*})(f\otimes g)$ \= $= G(F(f\otimes g))=
G(F_{2}(a',b')\circ (F(f) \otimes F(g)) \circ F_{2}^{-1}(a,b))$
\\[1ex]
\> $= G(F_{2}(a',b')) \cdot G(F(f) \otimes F(g)) \cdot
G(F_{2}^{-1}(a,b))$
\\[1ex]
\> $\stackrel{\mathrm{({\text {\tiny Lemma}}\ \ref{jedmat})}}{=\joinrel=\joinrel=\joinrel=\joinrel=\joinrel=}G(F(f) \otimes F(g))\stackrel{\mathrm{({\text {\tiny
Lemma}}\ \ref{Kroneker})}}{=\joinrel=\joinrel=\joinrel=\joinrel=\joinrel=}G(F(f)) \otimes
G(F(g))$
\\[1ex]
\> $=(GF^{*})(f)\otimes (GF^{*})(g)$.
\end{tabbing}
\end{proof}

In our next Proposition we show that the composition $GF^{*}$ maps symmetry to symmetry.
\begin{prop}\label{simet}{
$(GF^{*})(\tau_{a,b})=S_{n,m}.$}
\end{prop}
\begin{proof} From the commutativity of the following diagram
\begin{center}
\begin{picture}(120,60)(0,-5)

\put(0,40){\makebox(0,0){$F(a)\otimes F(b)$}}
\put(0,0){\makebox(0,0){$F(ab)$}}
\put(105,40){\makebox(0,0){$F(b)\otimes F(a)$}}
\put(105,0){\makebox(0,0){$F(ba)$}}

\put(53,47){\makebox(0,0){$\sigma_{F(a),F(b)}$}}
\put(50,-7){\makebox(0,0){$F(\tau_{a,b})$}}
\put(-20,20){\makebox(0,0){$F_{2}(a,b)$}}
\put(125,20){\makebox(0,0){$F_{2}(b,a)$}}

\put(35,40){\vector(1,0){30}} \put(30,0){\vector(1,0){40}}
\put(0,30){\vector(0,-1){20}} \put(105,30){\vector(0,-1){20}}

\end{picture}
\end{center}

we can see that
\beaz
(GF^{*})(\tau_{a,b})=G(F(\tau_{a,b}))=G(F_{2}(b,a) \circ
\sigma_{F(a),F(b)} \circ F_{2}^{-1}(a,b))=\,\,\\
G(F_{2}(b,a))\cdot G(\sigma_{F(a),F(b)}) \cdot
G(F_{2}^{-1}(a,b))\stackrel{\mathrm{({\text {\tiny Lemma}}\
\ref{jedmat})}}{=\joinrel=\joinrel=\joinrel=\joinrel=\joinrel=}G(\sigma_{F(a),F(b)}).\eeaz
Since the matrix representation of the linear map $\sigma_{F(a),F(b)}:F(a)\otimes
F(b)\rightarrow F(b) \otimes F(a)$ is independent of the choice of the bases
$e=[e_{0},\ldots, e_{n-1}]$ and $f=[f_{0},\ldots, f_{m-1}]$ for $F(a)$ and $F(b)$, we conclude that $G(\sigma_{F(a),F(b)})=S_{n,m}$.
\end{proof}

Note that we have actually proved that
$GF^{*}:1Cob\rightarrow Mat_{\mathbb{K}}$ is a strict $1$-TQFT. Clearly, $GF^{*}$ satisfies the condition of Corollary \ref{TQFT4}. Therefore, $GF^{*}$ is faithful.

\begin{cor}\label{strong} $F^{*}:1Cob\rightarrow VectB_{\mathbb{K}}$ is faithful.
\end{cor}
We can now formulate our main result.

\begin{thm}\label{strong}{Suppose $(F,F_{0},F_{2}):1Cob\rightarrow Vect_{\mathbb{K}}$ is a strong $1$-TQFT, mapping the null-dimensional manifold consisting of one point to a vector space of dimension at least $2$. Then $F$ is faithful.}
\end{thm}

\section{Acknowledgments}

The author is deeply indebted to Professor Zoran Petri\' c for having kindly suggested the problem and for his inspiring guidance and
encouragement throughout the progress of this work, without whose help this paper could not have taken this shape.
This study was supported by the Ministry of Education, Science, and
Technological Development of the Republic of Serbia (Grant 174032).


\begin{thebibliography}{99}

\bibitem{AM} {\sc M.\ Atiyah}, {\it Topological quantum field theories},
\textbf{\textit{Publications Math\'{e}matiques de l' I.H.\'{E}.S.}}, $\mathrm{vol.\ 68,\ (1988),\ pp.\ 175-186}$

\bibitem{BPT} {\sc Dj.\ Barali\' c, Z.\ Petri\' c and S.\ Telebakovi\' c},
{\it Spheres as Frobenius objects}, \textbf{\textit{Theory and Applications of Categories}}, $\mathrm{vol.\ 33,\ No.\ 24\ (2018),\
pp.\ 691-726}$

\bibitem{BR1} {\sc R.\ Brauer}, {\it On algebras which are connected with semisimple continuous groups},
\textbf{\textit{Annals of Mathematics}}, $\mathrm{vol.\ 38,\ (1937),\ pp.\ 857-872}$

\bibitem{DP1} {\sc K. Do\v sen and Z. Petri\' c}, {\it A Brauerian representation of split preorders},
\textbf{\textit{Mathematical Logic Quarterly}}, $\mathrm{vol.\ 49,\ (2003),\
pp.\ 579-586}$

\bibitem{DP2} --------, {\it Generality of proofs and its Brauerian representation},
\textbf{\textit{The Journal of Symbolic Logic}}, $\mathrm{vol.\ 68,\ (2003),\
pp.\ 740-750}$

\bibitem{DP3} --------, {\it Symmetric self-adjunctions:
A justification of Brauer's representation of Brauer's algebras},
\textbf{\textit{Proceedings of the Conference ``Contemporary
Geometry and Related Topics''}} $\mathrm{(N.\ Bokan\ et\ al.,\ editors),}$ $\mathrm{Faculty\ of\ Mathematics,}$ $\mathrm{Belgrade,}$ 2006, $\mathrm{pp.}$ 177-187

\bibitem{DP4} --------, {\it Symmetric self-adjunctions and matrices}, \textbf{\textit{Algebra Colloquium}}, vol.\ 19, No. spec01 (2012)
pp.\ 1051-1082

\bibitem{DF1} {\sc D.\ Dummit and R.\ Foote}, \textbf{\textit{Abstract Algebra, 3rd Edition,}}
$\mathrm{John\ Wiley\ \&\ Sons}$, $\mathrm{New\ York}$, 2004

\bibitem{J1} {\sc A.\ Jain}, \textbf{\textit{Fundamentals of Digital Image Processing}}, $\mathrm{Prentice-Hall,\ Inc.}$, $\mathrm{Engelwood\ Cliffs,\ New\ Jersey}$, 1989

\bibitem{J14} {\sc A.\ Juh\' asz}, {\it Defining and classifying TQFTs via surgery}, \textbf{\textit{Quantum Topology}}, $\mathrm{vol.\ 9}$ (2018), $\mathrm{pp.}$ 229-321

\bibitem{K1} {\sc J.\ Kock}, \textbf{\textit{Frobenius Algebras and 2D
Topological Quantum Field Theories}}, $\mathrm{Cambridge\ University\ Press}$, $\mathrm{Cambridge}$, 2003

\bibitem{L1} {\sc J. Lurie}, {\it On the Classification of Topological Field Theories}, \textbf{\textit{Current Developments in Mathematics}}, $\mathrm{vol.}$ (2008), $\mathrm{pp.}$ 129-280

\bibitem{ML1} {\sc S.\ Mac Lane}, \textbf{\textit{Categories for the Working
Mathematician}}, $\mathrm{Springer}$, $\mathrm{Berlin}$, 1971 ($\mathrm{expanded\ second\ edition}$,
1998)

\bibitem{MN} {\sc J.\ Magnus and H.\ Neudecker}, \textbf{\textit{Matrix Differential Calculus with Applications in Statistics and Econometrics, 2nd Edition,}} $\mathrm{John\ Wiley\ \&\ Sons}$, $\mathrm{Baffins\ Lane,\ Chichester,\
West\ Sussex}$, 1999

\bibitem{PT17} {\sc Z.\ Petri\' c and S.\ Telebakovi\' c  Oni\' c}, \textbf{\textit{ A Faithful 2-dimensional TQFT}}, $\mathrm{arXiv:1711.06044}$, 2017

\bibitem{FK} {\sc F.\ Quinn}, {\it Lectures on axiomatic topological quantum field
theory}, \textbf{\textit{Geometry and Quantum Field Theory}}
(D.S.\ Freed and K.K.\ Uhlenbeck, editors), American Mathematical
Society, Providence, 1995, pp.\ 323-453

\bibitem{crnac} {\sc C.\ Rakotonirina}, \textbf{\textit{On the Tensor Permutation Matrices}}, $\mathrm{arXiv:1101.0910v3}$, 2013

\bibitem{V} {\sc W.J.\ Vetter}, {\it Vector Structures and Solutions of Linear Matrix Equations}, \textbf{\textit{Linear Algebra and Its Applications}}, $\mathrm{vol.\ 10,\ (1975),\
pp.\ 181-188}$


\bibitem{EW} {\sc E.\ Witten}, {\it Topological quantum field theory}, \textbf{\textit{Communications in Mathematical Physics}}, $\mathrm{vol.\ 117,\ (1988),\
pp.\ 353-386}$

\end{thebibliography}
\end{document}